\documentclass[11pt]{amsart}

\usepackage{amsmath, amsrefs, amssymb, amsthm, thmtools, mathtools, enumitem, cancel, extarrows, accents, parskip}
\usepackage[cal=rsfso, calscaled=1.03]{mathalpha}
\usepackage[all]{xy}
\usepackage[margin=1.5in]{geometry}

\usepackage{xpatch}
\patchcmd{\section}{\normalfont\scshape}{\bfseries}{}{}

\declaretheoremstyle[bodyfont=\itshape, headformat={\makebox[0pt][r]{\makebox[15mm][l]{\texttt{\NUMBER}}}\NAME\NOTE}]{stythm}
\declaretheoremstyle[headformat={\makebox[0pt][r]{\makebox[15mm][l]{\texttt{\NUMBER}}}\NAME\NOTE}]{stydfn}
\declaretheoremstyle[headformat={\makebox[0pt][r]{\makebox[15mm][l]{\texttt{\NUMBER}}}\NAME\NOTE}]{styclm}

\makeatletter
\let\mytagform@=\tagform@
\def\tagform@#1{\maketag@@@{{\makebox[0pt][r]{\makebox[15mm][l]{\ttfamily\ignorespaces#1\unskip\@@italiccorr}}}}\hspace{3mm}}
\renewcommand{\eqref}[1]{\textup{\mytagform@{\ref{#1}}}}
\makeatother

\declaretheorem[name=Theorem, numberwithin=section, style=stythm]{theorem}
\declaretheorem[name=Lemma, sibling=theorem, style=stythm]{lemma}

\declaretheorem[name=Corollary, sibling=theorem, style=stythm]{corollary}
\declaretheorem[name=Lemma-Definition, sibling=theorem, style=stythm]{deflemma}

\declaretheorem[name=Fact, sibling=theorem, style=stydfn]{fact}
\declaretheorem[name=Definition, sibling=theorem, style=stydfn]{definition}
\declaretheorem[name=Notation, sibling=theorem, style=stydfn]{notation}
\declaretheorem[name=Remark, sibling=theorem, style=stydfn]{remark}

\declaretheorem[numbered=no, name=Theorem A]{thmA}
\declaretheorem[numbered=no, name=Theorem B]{thmB}
\declaretheorem[numbered=no, name=Theorem C]{thmC}

\numberwithin{equation}{theorem}

\declaretheorem[name=Claim, sibling=equation, style=styclm]{claim}

\DeclareEmphSequence{\bfseries}

\DeclareMathOperator{\im}{im}
\renewcommand{\emptyset}{{\cancel{\mathrm{o}}}}
\newcommand{\op}{{\mathrm{op}}}

\title{On interval decomposition of persistence modules}

\author{Wee Liang Gan}
\address{Department of Mathematics, University of California, Riverside, CA 92521, USA}
\email{wlgan@ucr.edu}

\author{Nadiya Upegui Keagy}
\address{Department of Mathematics, University of California, Riverside, CA 92521, USA}
\email{nadiya.upegui@email.ucr.edu}

\begin{document}

\begin{abstract}
   We give down-to-earth proofs of the structure theorems for persistence modules.
\end{abstract}

\maketitle

\section*{Introduction}

In this article, we give alternative proofs of the following theorems.

\begin{thmA}
   Let $V$ be a persistence module indexed by a poset. Then interval decompositions of $V$, if they exist, are unique up to automorphisms of $V$.
\end{thmA}

\begin{thmB}
   Let $V$ be a pointwise finite dimensional persistence module indexed by a totally ordered set. Then $V$ has an interval decomposition.
\end{thmB}

\begin{thmC}
   Let $V$ be a pointwise finite dimensional persistence module indexed by a zigzag poset. Then $V$ has an interval decomposition.
\end{thmC}

Theorem A is a special case of the Krull-Remak-Schmidt-Azumaya Theorem, see for example \cite{par}*{\S 4.8}. We give a simple direct proof of it in Section \ref{sec: interval decomposition}; see Theorem \ref{thm:uniqueness}.

Theorem B was proved by Botnan and Crawley-Boevey in \cite{bot-cb}; we refer the reader to \cite{ou}*{\S 1.5} and \cite{po}*{\S 2.1} for discussions of earlier results. Our proof of Theorem B results from our attempt to understand a sketch of its proof given by Gabriel and Roiter in \cite{gr}*{\S 3.6}. We give our proof of Theorem B in Section \ref{sec: totally ordered}; see Theorem \ref{thm:totally ordered case}. 
 
The definition of a zigzag poset is given in Section \ref{sec: zigzag generalities}; see Definition \ref{def:zigzag}. An example of a zigzag poset is the poset $Z(\gamma)$ studied by Botnan and Crawley-Boevey in \cite{bot-cb}*{\S5.1} for which they proved Theorem C; see also \cite{ig}. 

In Section \ref{sec: zigzag finite}, we prove Theorem C when the zigzag poset has only finitely many extrema; see Theorem \ref{thm:finite zigzag}.

In Section \ref{sec: zigzag infinite},  we prove Theorem C when the zigzag poset has infinitely many extrema; see Theorem \ref{thm: zigzag N case} and Theorem \ref{thm:zigzag Z case}. 

This article is essentially self-contained; we use only elementary results in linear algebra and Zorn's lemma.

\section{Definitions and notation} \label{sec:definitions}

Fix a field $k$. Let $P$ be a poset. By a \emph{persistence module} indexed by $P$, we mean a functor from $P$ regarded as a category to the category of $k$-vector spaces. A \emph{morphism} of persistence modules indexed by $P$ is, by definition, a natural transformation of functors.

\begin{notation}
   Let $P$ be a poset. Let $V$ be any persistence module indexed by $P$. We write $V_x$ for the $k$-vector space which $V$ assigns to $x\in P$. We write $V_{x,y}: V_x \to V_y$ for the $k$-linear map which $V$ assigns to $x,y\in P$ when $x\leq y$. 
\end{notation}

\begin{notation}
   Let $P$ be a poset. Let $f:V\to W$ be any morphism of persistence modules indexed by $P$. For each $x\in P$, we write $f_x: V_x\to W_x$ for the component of $f$ at $x$.
\end{notation}

\begin{notation} \label{nota: in and pr}
   Let $P$ be a poset. Let $V$ be a persistence module indexed by $P$. Let $\mathcal{A}$ be a set of submodules of $V$ giving an internal direct sum decomposition
   \[ V=\bigoplus_{U\in \mathcal{A}} U. \] 
   Then for each $U\in \mathcal{A}$, we write $\mathrm{in}^U: U\to V$ and $\mathrm{pr}^U : V \to U$ for the inclusion and projection, respectively.
\end{notation}

\begin{definition}
   Let $P$ be a poset. Let $V$ be a persistence module indexed by $P$. We say that $V$ is \emph{pointwise finite dimensional} if $V_x$ is finite dimensional for all $x\in P$. 
\end{definition}

\begin{definition}
   Let $P$ be a poset. Let $I$ be a subset of $P$. 
   \begin{enumerate}
      \item We say that $I$ is \emph{convex} if: for all $x,y,z\in P$, we have $y\in I$ whenever $x,z\in I$ and $x\leq y\leq z$. 
      \item We say that $I$ is \emph{connected} if: for all $x,z\in I$, there exist $y_0, y_1, \ldots, y_n\in I$ such that
   \begin{equation} \label{eq:connected}
      \begin{cases}
         y_0=x, \\
         y_n=z, \\
         y_{i-1}\leq y_i \enspace \mbox{ or } \enspace y_{i-1}\geq y_i \quad \mbox{ for each }i\in \{1, \ldots, n\}.
      \end{cases}
   \end{equation} 
   \item We call $I$ an \emph{interval} in $P$ if $I$ is nonempty, convex, and connected.  
   \end{enumerate}
\end{definition}

\begin{definition} \label{def:constant module}
   Let $P$ be a poset. Let $I$ be an interval in $P$. The \emph{constant module} $M(I)$ is the persistence module indexed by $P$ defined as follows:
\begin{enumerate}
   \item for all $x\in P$,
   \[ M(I)_x = \begin{cases}
      k & \mbox{ if }x\in I,\\
      0 & \mbox{ if }x\notin I;
      \end{cases} \] 
   \item for all $x\leq y$ in $P$,
   \[ M(I)_{x,y} = \begin{cases}
      \mathrm{id}_k & \mbox{ if }x,y\in I,\\
      0 & \mbox{ if }x\notin I \mbox{ or }y\notin I.
      \end{cases} \]
\end{enumerate}
\end{definition}

\begin{definition}
   Let $P$ be a poset. An \emph{interval module} indexed by $P$ is a persistence module indexed by $P$ which is isomorphic to a constant module.
\end{definition}

\begin{remark}
   The zero persistence module indexed by a poset is not an interval module.
\end{remark}

\begin{definition}
   Let $P$ be a poset. Let $V$ be a persistence module indexed by $P$. An \emph{interval decomposition} of $V$ is an internal direct sum decomposition 
   \[ V= \bigoplus_{U\in \mathcal{A}} U \]
   where $\mathcal{A}$ is a (possibly empty) set of submodules of $V$ such that every $U\in \mathcal{A}$ is an interval module.
\end{definition}

\section{Reminder on constant modules} \label{sec:constant modules}

\begin{lemma} \label{lem:endo}
   Let $P$ be a poset. Let $I$ be an interval in $P$. Let $f: M(I)\to M(I)$ be a morphism. Then there exists a unique $\lambda\in k$ such that $f=\lambda\cdot \mathrm{id}_{M(I)}$. 
\end{lemma}

\begin{proof}
   For each $x\in I$, there exists a unique $\lambda_x\in k$ such that $f_x = \lambda_x \cdot \mathrm{id}_k$. We need to prove that for all $x,z\in I$, we have $\lambda_x = \lambda_z$.

   Observe that for all $w,y\in I$ such that $w\leq y$, we have 
   \[ M(I)_{w,y} \circ f_w  = f_y \circ M(I)_{w,y},\] 
   which implies that $\lambda_w = \lambda_y$. 

   Now consider any $x,z\in I$. Since $I$ is connected, there exist $y_0, y_1, \ldots, y_n\in I$ such that \eqref{eq:connected} holds. By the above observation, we have $\lambda_{y_0} = \lambda_{y_1} = \cdots = \lambda_{y_n}$, hence $\lambda_x=\lambda_z$.
\end{proof}

\begin{corollary}
   Let $P$ be a poset. Let $I$ be an interval in $P$. Then $M(I)$ is indecomposable. 
\end{corollary}

\begin{proof}
   Suppose that $M(I)=U\oplus W$, where $U, W$ are submodules of $M(I)$. Define $f:M(I)\to M(I)$ by 
   \[ f=\mathrm{in}^U \circ \mathrm{pr}^U \] 
   (see Notation \ref{nota: in and pr}).
   We have $f\circ f=f$. By Lemma \ref{lem:endo}, we also have $f=\lambda \cdot \mathrm{id}_{M(I)}$ for some $\lambda\in k$. Observe that:
   \[ f\circ f= f \quad\Longrightarrow\quad \lambda^2=\lambda \quad\Longrightarrow\quad \lambda=0\mbox{ or }\lambda=1 
      \quad\Longrightarrow\quad U=0 \mbox{ or }W=0. \]
\end{proof}

\section{Interval decomposition} \label{sec: interval decomposition}

In this section, we prove Theorem A and a criterion for the existence of an interval decomposition.

\begin{lemma} \label{lem:zero}
   Let $P$ be a poset. Let $I$ and $J$ be distinct intervals in $P$. Let $g: M(I)\to M(J)$ and $h:M(J)\to M(I)$ be morphisms. Then $h\circ g=0$.
\end{lemma}

\begin{proof}
   First, by Lemma \ref{lem:endo}, there exists a $\lambda\in k$ such that 
   \begin{equation} \label{eq: hg}
      h\circ g = \lambda \cdot \mathrm{id}_{M(I)}.
   \end{equation} 
   Next, since $I\neq J$, we have $I- J\neq \emptyset$ or $J- I\neq \emptyset$.

   \textbf{Case 1}: $I- J\neq \emptyset$.

   Let $x\in I- J$. Then:
   \begin{quote}
      $x\in I$, so by \eqref{eq: hg} we have $h_x\circ g_x = \lambda \cdot \mathrm{id}_k$;
      
      $x\notin J$, so $M(J)_x=0$, which implies that $g_x=0$ and $h_x=0$. 
   \end{quote}
   It follows that $\lambda=0$, so by \eqref{eq: hg} we have $h\circ g=0$. 
   
   \textbf{Case 2}: $J- I\neq \emptyset$. 

   We have $g\circ h=0$ from Case 1 by symmetry. This implies that $g\circ h\circ g=0$, hence by \eqref{eq: hg} we have $\lambda \cdot g=0$. 

   Obviously if $g=0$, then $h\circ g=0$. On the other hand, if $g\neq 0$, then $\lambda \cdot g=0$ implies that $\lambda=0$, so by \eqref{eq: hg} we again have $h\circ g=0$. 
\end{proof}

\begin{lemma} \label{lem:zero2}
   Let $P$ be a poset. Let $U, W, Z$ be interval modules indexed by $P$. Let $g: U \to W$ and $h:W \to Z$ be morphisms. Assume that 
   $W\ncong U$ and $Z\cong U$. Then $h\circ g=0$.
\end{lemma}

\begin{proof}
   Immediate from Lemma \ref{lem:zero}.
\end{proof}

\begin{lemma} \label{lem: intersection}
   Let $P$ be a poset. Let $V$ be a persistence module indexed by $P$. Let 
   \[ V= \bigoplus_{U\in \mathcal{A}} U \quad \mbox{ and }\quad V = \bigoplus_{W\in \mathcal{B}} W \]
   be interval decompositions of $V$. Let $I$ be an interval in $P$. Let 
   \begin{align*}
   \mathcal{C} &= \{U\in \mathcal{A} \mid U\cong M(I)\}, \\ 
   \mathcal{D} &= \{W\in \mathcal{B} \mid W \ncong M(I) \}.  
   \end{align*} 
   Then 
   \[ \left( \bigoplus_{U\in \mathcal{C}} U \right) \cap \left( \bigoplus_{W\in \mathcal{D}} W \right) = 0. \]
\end{lemma}

\begin{proof}
   Let $x\in P$. Let 
   \[ v \in  \left( \bigoplus_{U\in \mathcal{C}} U_x \right) \cap \left( \bigoplus_{W\in \mathcal{D}} W_x \right). \]
   We need to show that $v=0$.

   There exists a finite subset $\mathcal{C}(v)$ of $\mathcal{C}$ such that
   \begin{equation} \label{eq: v_U}
      v = \sum_{U\in \mathcal{C}(v)} (\mathrm{in}^U_x \circ \mathrm{pr}^U_x)(v)
   \end{equation}
   (see Notation \ref{nota: in and pr}).
   There also exists a finite subset $\mathcal{D}(v)$ of $\mathcal{D}$  such that
   \begin{equation} \label{eq: v_W}
      v = \sum_{W\in \mathcal{D}(v)} (\mathrm{in}^W_x \circ \mathrm{pr}^W_x)(v). 
   \end{equation}   
   Let 
   \begin{align*}
      g^{U, W} &= \mathrm{pr}^W \circ \mathrm{in}^U \qquad\mbox{ for all }U\in \mathcal{A},\; W\in \mathcal{B};\\ 
      h^{W, Z} &= \mathrm{pr}^Z \circ \mathrm{in}^W \qquad\mbox{ for all }W\in \mathcal{B},\; Z\in \mathcal{A}.
   \end{align*}
   Then
   \begin{align*}
      v &= \sum_{Z\in \mathcal{C}(v)} (\mathrm{in}^Z_x \circ \mathrm{pr}^Z_x)(v) & \mbox{by \eqref{eq: v_U}}\\
      &= \sum_{Z\in \mathcal{C}(v)} \sum_{W\in \mathcal{D}(v)} (\mathrm{in}^Z_x \circ \mathrm{pr}^Z_x \circ \mathrm{in}^W_x \circ \mathrm{pr}^W_x)(v) & \mbox{by \eqref{eq: v_W}}\\
      &= \sum_{Z\in \mathcal{C}(v)} \sum_{W\in \mathcal{D}(v)}  \sum_{U\in \mathcal{C}(v)} (\mathrm{in}^Z_x \circ \mathrm{pr}^Z_x \circ \mathrm{in}^W_x \circ \mathrm{pr}^W_x \circ \mathrm{in}^U_x \circ \mathrm{pr}^U_x)(v) 
      & \mbox{by \eqref{eq: v_U}}\\
      &= \sum_{Z\in \mathcal{C}(v)} \sum_{W\in \mathcal{D}(v)}  \sum_{U\in \mathcal{C}(v)} (\mathrm{in}^Z_x \circ h^{W,Z}_x \circ g^{U,W}_x \circ \mathrm{pr}^U_x)(v) & \\
      &= 0 & \mbox{by Lemma \ref{lem:zero2}.}
   \end{align*}
\end{proof}

We now restate and prove Theorem A.

\begin{theorem} \label{thm:uniqueness}
   Let $P$ be a poset. Let $V$ be a persistence module indexed by $P$. Let 
   \[ V= \bigoplus_{U\in \mathcal{A}} U \quad \mbox{ and }\quad V = \bigoplus_{W\in \mathcal{B}} W \]
   be interval decompositions of $V$. Let $I$ be an interval in $P$. Let 
   \begin{align*} 
      \mathcal{C} &= \{U\in \mathcal{A} \mid U \cong M(I)\}, \\ 
      \mathcal{C}'&=\{W\in \mathcal{B} \mid W\cong M(I)\}. 
   \end{align*} 
   Then $|\mathcal{C}|=|\mathcal{C}'|$. 
\end{theorem}

\begin{proof}
   Let 
   \[ C =\bigoplus_{U\in \mathcal{C}} U, \qquad D = \bigoplus_{U\in \mathcal{A}-\mathcal{C}} U, \qquad C' = \bigoplus_{W\in \mathcal{C}'} W, \qquad D'=\bigoplus_{W\in \mathcal{B}-\mathcal{C}'} W. \]
   Then $V= C\oplus D$, and also, $V= C'\oplus D'$.
   
   We have $\mathrm{in}^C: C\to V$ and $\mathrm{pr}^{C'} : V\to C'$ (see Notation \ref{nota: in and pr}). The kernel of $\mathrm{pr}^{C'}$ is $D'$. 
   
   Define $f: C \to C'$ by 
   \[ f= \mathrm{pr}^{C'} \circ \mathrm{in}^C. \] 
   By Lemma \ref{lem: intersection}, we have $C\cap D'=0$. This implies that $f$ is injective. 

   Pick $x\in I$. Then $f_x: C_x \to C'_x$ is an injective $k$-linear map. Hence
   \[ \dim C_x \leq \dim C'_x. \] 
   Since $\dim C_x = |\mathcal{C}|$ and $\dim C'_x = |\mathcal{C}'|$, we have 
   \[ |\mathcal{C}| \leq |\mathcal{C}'|.\]
   By symmetry we also have $|\mathcal{C}'| \leq |\mathcal{C}|$. Therefore $|\mathcal{C}|=|\mathcal{C}'|$. 
\end{proof}

\begin{remark}
   The sets $\mathcal{C}$ and $\mathcal{C}'$ in Theorem \ref{thm:uniqueness} may be infinite. 
\end{remark}

We next apply Zorn's lemma to prove a decomposition criterion.

\begin{lemma} \label{lem:decomposition criteria}
   Let $P$ be a poset. Let $V$ be a pointwise finite dimensional persistence module indexed by $P$. Let $\mathcal{E}$ be a set of nonzero submodules of $V$.
   Assume that for each nonzero direct summand $W$ of $V$, there exists a direct summand $U$ of $W$ such that $U\in \mathcal{E}$. Then there exists a subset $\mathcal{A}$ of $\mathcal{E}$ giving an internal direct sum decomposition $V=\bigoplus_{U\in \mathcal{A}}U$. 
\end{lemma}

\begin{proof}
   Let $\mathbf{S}$ be the set of all pairs $(\mathcal{A}, W)$ where: 
   \begin{quote}
      $\mathcal{A}$ is a subset of $\mathcal{E}$,  
      
      $W$ is a submodule of $V$,
   \end{quote}
   such that we have an internal direct sum decomposition $V=\left(\bigoplus_{U\in \mathcal{A}} U \right) \oplus W$. 

   The set $\mathbf{S}$ is nonempty because $(\emptyset, V)\in \mathbf{S}$.

   Define a partial order on $\mathbf{S}$ by $(\mathcal{A}, W) \leq (\mathcal{B}, Z)$ if $\mathcal{A} \subseteq \mathcal{B}$ and
   \[ W = \left( \bigoplus_{U \in \mathcal{B} - \mathcal{A}} U  \right) \oplus Z. \] 
   
   \begin{claim} \label{claim:decomposition upper bound}
      Every chain in $\mathbf{S}$ has an upper bound. More precisely:
   
      \begin{quote} 
         For any chain $\mathbf{T}$ in $\mathbf{S}$, let 
      \begin{align*} 
         \mathcal{C} = \bigcup_{(\mathcal{A}, W)\in \mathbf{T}} \mathcal{A}, \qquad
         N = \bigcap_{(\mathcal{A}, W)\in \mathbf{T}} W.
      \end{align*}
      Then $(\mathcal{C}, N)$ lies in $\mathbf{S}$ and is an upper bound of $\mathbf{T}$.
     \end{quote}
   \end{claim}
     
   \begin{proof}[Proof of Claim \ref{claim:decomposition upper bound}]
      There are two things to check: 
      \begin{enumerate}
         \item $V= \left(\bigoplus_{U\in \mathcal{C}} U\right) \oplus N$; 
      
         \item $W=\left( \bigoplus_{U\in \mathcal{C}-\mathcal{A}} U\right) \oplus N$ for all $(\mathcal{A}, W)\in \mathbf{T}$.
      \end{enumerate}
      Thus for each $x\in P$, we need to check:
      \begin{enumerate}
         \item $V_x= \left(\bigoplus_{U\in \mathcal{C}} U_x\right) \oplus N_x$; 
      
         \item $W_x=\left( \bigoplus_{U\in \mathcal{C}-\mathcal{A}} U_x\right) \oplus N_x$ for all $(\mathcal{A}, W)\in \mathbf{T}$.
      \end{enumerate}

      Fix $x\in P$ and choose a pair $(\mathcal{B}, Z)\in \mathbf{T}$ such that $\dim Z_x$ is minimal. 
      
      Now let $(\mathcal{A}, W)\in \mathbf{T}$. 
      
      We make the following observations:
      \begin{enumerate}[label=(\alph*)]
         \item If $(\mathcal{A}, W) \leq (\mathcal{B}, Z)$, then 
         \[ W_x = \left( \bigoplus_{U \in \mathcal{B} - \mathcal{A}} U_x  \right) \oplus Z_x; \] 
         hence $Z_x \subseteq W_x$. 
         
         \item If $(\mathcal{B}, Z) \leq (\mathcal{A}, W)$, then 
         \[ Z_x = \left( \bigoplus_{U \in \mathcal{A} - \mathcal{B}} U_x  \right) \oplus W_x; \] 
         by minimality of $\dim Z_x$, it follows that $Z_x=W_x$, and $U_x = 0$ for all $U\in \mathcal{A}-\mathcal{B}$. 

         \item By (a) and (b), we have $N_x = Z_x$, and $U_x=0$ for all $U\in \mathcal{C}-\mathcal{B}$. 
      \end{enumerate}

      We deduce that:
      \begin{align*}
         V_x &= \left( \bigoplus_{U\in \mathcal{B}} U_x \right) \oplus Z_x & \\
         &= \left( \bigoplus_{U\in \mathcal{C}} U_x \right) \oplus N_x & \mbox{by (c).}
      \end{align*}
      If  $(\mathcal{A}, W) \leq (\mathcal{B}, Z)$, then
      \begin{align*}
         W_x &= \left( \bigoplus_{U \in \mathcal{B} - \mathcal{A}} U_x  \right) \oplus Z_x & \\
         &= \left( \bigoplus_{U \in \mathcal{C} - \mathcal{A}} U_x  \right) \oplus N_x & \mbox{by (c).}
      \end{align*}
      If $(\mathcal{B}, Z) \leq (\mathcal{A}, W)$, then
      \begin{align*}
         W_x &= Z_x & \mbox{by (b)} \\
         &=  \left( \bigoplus_{U \in \mathcal{C} - \mathcal{A}} U_x  \right) \oplus N_x & \mbox{by (c).}
      \end{align*}
      Claim \ref{claim:decomposition upper bound} is proved.
   \end{proof}

   By Claim \ref{claim:decomposition upper bound}, we can apply Zorn's lemma to deduce that there exists a pair, say $(\mathcal{A}, W)$, which is maximal in $\mathbf{S}$.
   
   Suppose that $W\neq 0$. Then by assumption there exists an internal direct sum decomposition $W=U \oplus Z$ where $U\in \mathcal{E}$. Let $\mathcal{B}=\mathcal{A} \cup \{U\}$. Then we have $(\mathcal{B}, Z)\in \mathbf{S}$ and $(\mathcal{A}, W) < (\mathcal{B}, Z)$, a contradiction to the  maximality of $(\mathcal{A}, W)$. 
   
   Therefore $W=0$. Hence $V= \bigoplus_{U\in \mathcal{A}} U$ is a desired decomposition.
   \end{proof}

\begin{lemma} \label{lem:interval decomposition criteria}
   Let $P$ be a poset. Assume that every nonzero pointwise finite dimensional persistence module indexed by $P$ has a direct summand which is an interval module. Then every pointwise finite dimensional persistence module indexed by $P$ has an interval decomposition.
\end{lemma}

\begin{proof}
   Let $V$ be a pointwise finite dimensional persistence module indexed by $P$. We can apply Lemma \ref{lem:decomposition criteria} by taking $\mathcal{E}$ to be the set of submodules of $V$ which are interval modules. The result follows.
\end{proof}

\section{Reminder on linear duality} \label{sec: linear duality}

To set our notation and for ease of reference, we collect in this section some basic facts from linear algebra.

\begin{notation} \label{nota:linear algebra}
   Let $E$ be any $k$-vector space. 
   \begin{enumerate}
      \item We write $E^*$ for the linear dual of $E$. 
      \item We write $\mathcal{S}(E)$ for the poset of linear subspaces of $E$ (where the partial order is inclusion).
      \item For each $D\in \mathcal{S}(E)$, we set $D^\perp = \{\alpha \in E^* \mid \alpha(v)=0 \mbox{ for all }v\in D\}$.      
      \item For each $\mathfrak{D} \in \mathcal{S}(E^*)$, we set ${^\perp\mathfrak{D}} = \{v\in E \mid \alpha(v)=0 \mbox{ for all }\alpha\in \mathfrak{D}\}$. 
   \end{enumerate}
\end{notation}

\begin{fact} \label{fact:direct sum dual}
      Let $E$ be a finite dimensional $k$-vector space. Then
      \begin{enumerate}
         \item we have mutually inverse bijections: 
      \begin{align*}
         \mathcal{S}(E) \to \mathcal{S}(E^*), \qquad & D \mapsto D^\perp; \\
         \mathcal{S}(E^*) \to \mathcal{S}(E), \qquad & \mathfrak{D} \mapsto {^\perp\mathfrak{D}}; 
      \end{align*}
      \item the bijections in (1) are order-reversing;
      \item for any $D\in \mathcal{S}(E)$, there is a natural bijective $k$-linear map 
      \[ (E/D)^* \to D^\perp, \qquad \theta \mapsto \theta \circ \pi, \]
      where $\pi:E\to E/D$ is the quotient map.
      \item for any $\mathfrak{D}_1, \mathfrak{D}_2 \in \mathcal{S}(E^*)$:
      \[ E^* = \mathfrak{D}_1\oplus \mathfrak{D}_2 \quad\Longrightarrow\quad E = {^\perp\mathfrak{D}_1} \oplus {^\perp\mathfrak{D}_2}. \] 
      \end{enumerate}
   \end{fact}

\begin{notation}
   For any $k$-vector spaces $E_1, E_2$, and $k$-linear map $f: E_1\to E_2$, we write 
   \[ {^t f}: E_2^* \to E_1^* \] 
   for the transpose of $f$. 
\end{notation}

\begin{fact} \label{fact:transpose}
   Let $E_1, E_2$ be $k$-vector spaces. Let $f: E_1\to E_2$ be a $k$-linear map. Then
   \begin{enumerate}
      \item we have $\im({^t f}) = (\ker f)^\perp$;
      \item for any $\mathfrak{D}_1 \in \mathcal{S}(E_1^*)$ and $\mathfrak{D}_2\in \mathcal{S}(E_2^*)$: 
         \[ {^t f}(\mathfrak{D}_2) \subseteq \mathfrak{D}_1 \quad \Longrightarrow \quad f({^\perp \mathfrak{D}_1}) \subseteq {^\perp \mathfrak{D}_2}. \]
   \end{enumerate}   
\end{fact}

\section{Reminder on flags} \label{sec: flags}

\begin{definition}
   Let $E$ be a finite dimensional $k$-vector space.
   \begin{enumerate}
      \item A \emph{flag} in $E$ is a chain in $\mathcal{S}(E)$ (see Notation \ref{nota:linear algebra}(2)). 
      \item A \emph{complete flag} in $E$ is a maximal chain in $\mathcal{S}(E)$.
   \end{enumerate} 
\end{definition} 

\begin{definition}
      Let $E$ be a finite dimensional $k$-vector space. Let $\mathcal{F}$ be a flag in $E$. We say that a basis $B$ of $E$ is \emph{compatible} with $\mathcal{F}$ if: for each $F\in \mathcal{F}$, the set $B\cap F$ is a basis of $F$.  
\end{definition}

The following lemma is well-known and appears to be due to Steinberg, see \cite{st}*{Lemma 2.1}. For the reader's convenience, we include Steinberg's proof.

\begin{lemma} \label{lem:basis for two flags}
   Let $E$ be a finite dimensional $k$-vector space. Let $\mathcal{F}$ and $\mathcal{G}$ be flags in $E$. Then there exists a basis of $E$ which is compatible with both $\mathcal{F}$ and $\mathcal{G}$. 
\end{lemma}

\begin{proof}
   Let $n=\dim E$. We use induction on $n$, the case $n=0$ being trivial. 
   
   Assume $n\geq 1$. We may assume further that $\mathcal{F}$ and $\mathcal{G}$ are complete flags, say 
   $\mathcal{F} = \{F_0, F_1, \ldots, F_n\}$ and $\mathcal{G} = \{G_0, G_1, \ldots, G_n\}$ 
   where $\dim F_i = \dim G_i = i$ for each $i$. 
   
   Observe that there exists a unique $s$ such that $G_s \subseteq F_{n-1}$ but $G_{s+1} \nsubseteq F_{n-1}$.
   
   For each $i\in\{0, 1, \ldots, n-1\}$, let 
      \[ G'_i = \begin{cases}
         G_i & \mbox{ if } i\leq s,\\
         G_{i+1} \cap F_{n-1} & \mbox{ if }i>s.
         \end{cases} \]

   Let $\mathcal{F}' = \{F_0, F_1, \ldots, F_{n-1}\}$ and $\mathcal{G}' = \{G'_0, G'_1, \ldots, G'_{n-1}\}$.

   Obviously $\mathcal{F}'$ is a complete flag in $F_{n-1}$. A straightforward calculation yields $\dim G'_i = i$ for each $i$. Thus $\mathcal{G}'$ is also a complete flag in $F_{n-1}$. By induction hypothesis, there exists a basis $B'$ of $F_{n-1}$ which is compatible with both $\mathcal{F}'$ and $\mathcal{G}'$. 
   
   Choose any $e\in G_{s+1}-F_{n-1}$.
   
   Let $B=B'\cup\{e\}$. It is easy to see that $B$ is a basis of $E$ which is compatible with both $\mathcal{F}$ and $\mathcal{G}$. 
\end{proof}

\section{Totally ordered sets} \label{sec: totally ordered}

The aim of this section is to prove Theorem B. 

We begin with the following lemma which is a special case of \cite{bour}*{Chapter III, \S 7.4, Example 2}. The proof we give is a slight simplification of the proof of \cite{bour}*{Chapter III, \S 7.4, Theorem 1}.

\begin{lemma} \label{lem:inverse limit}
   Let $P$ be a totally ordered set. Let $V$ be a pointwise finite dimensional persistence module indexed by $P$. Let 
   \[ z\in P,\qquad v\in V_z,\qquad  I = \left\{ x \in P \mid x\leq z,\ V_{xz}^{-1}(v)\neq \emptyset \right\}.\]
 Then there exists $(u_x)\in \prod_{x\in I} V_x$ such that
   \begin{equation} \label{eq:inverse limit}
      \begin{cases} 
      V_{xz}(u_x) = v \quad \mbox{ for all } x\in I,\\
      V_{xy}(u_x) = u_y \quad \mbox{ for all } x,y\in I,\ x\leq y.   
      \end{cases}
   \end{equation}
\end{lemma}

\begin{proof}
   Let $\mathbf{S}$ be the set of families $A=\{A_x\}$, indexed by $x\in I$, such that each $A_x$ is an affine subspace of $V_x$ and the following conditions are satisfied:
   \begin{equation} \label{eq:conditions on family}
      \begin{cases}
         V_{xz}(A_x)=\{v\} \quad \mbox{ for all }x\in I; \\

         V_{xy}(A_x)\subseteq A_y \quad \mbox{ for all } x,y\in I,\ x\leq y.
      \end{cases}
   \end{equation}
   The set $\mathbf{S}$ is nonempty since it contains the family $E=\{E_x\}$, where $E_x=V_{xz}^{-1}(v)$ for each $x\in I$.

   In the rest of the proof, we write $A, B$ for families $\{A_x\}, \{B_x\}$, respectively.

   Define a partial order on $\mathbf{S}$ by 
   \[ B \leq A \quad \mbox{ if} \quad A_x \subseteq B_x \mbox{ for each }x\in I. \] 

  \begin{claim} \label{claim:inverse limit upper bound}
      Every chain in $\mathbf{S}$ has an upper bound. More precisely:
   
      \begin{quote} 
         Let $\mathbf{T}$ be a chain in $\mathbf{S}$. Define a family $A$ by
         \[ A_x = \bigcap_{B\in \mathbf{T}} B_x  \quad \mbox{ for each }x \in I. \]           
      Then $A$ lies in $\mathbf{S}$ and is an upper bound of $\mathbf{T}$.
      \end{quote}
   \end{claim}

   \begin{proof}[Proof of Claim \ref{claim:inverse limit upper bound}]
      Fix $x\in I$. Choose $B\in \mathbf{T}$ such that $\dim B_x$ is minimal. Then $A_x=B_x$, which implies that $A_x$ is a (nonempty) affine subspace of $V_x$. 
      
      It is easy to see that $A$ satisfies the conditions in \eqref{eq:conditions on family}. Clearly $A$ is an upper bound of $\mathbf{T}$. Claim \ref{claim:inverse limit upper bound} is proved.
   \end{proof}

   \begin{claim} \label{claim:maximal surjective}
      Let $A\in \mathbf{S}$. Assume that $A$ is maximal in $\mathbf{S}$. Then 
      \[ V_{x,y}(A_x)=A_y \quad \mbox{ for all } x,y\in I,\ x\leq y. \]    
   \end{claim}

   \begin{proof}[Proof of Claim \ref{claim:maximal surjective}]
      Fix $x\in I$. Define a family $B$ by 
      \[ B_y = \begin{cases}
         A_y & \mbox{ if } y\leq x,\\
         V_{x,y}(A_x) & \mbox{ if } y>x.      
      \end{cases} \]
      Then $B$ lies in $\mathbf{S}$. Clearly $A\leq B$. By maximality of $A$, we have $B=A$; in particular $B_y=A_y$ for all $y\in I$, hence 
      \[  V_{x,y}(A_x)=A_y \quad \mbox{ for all } y\in I,\ x\leq y. \]
      Claim \ref{claim:maximal surjective} is proved. 
   \end{proof}

   \begin{claim} \label{claim:maximal singleton}
      Let $A\in \mathbf{S}$. Assume that $A$ is maximal in $\mathbf{S}$. Then $|A_x|=1$ for all $x\in I$.
   \end{claim}

   \begin{proof}[Proof of Claim \ref{claim:maximal singleton}]
      Fix $x\in I$. Choose any element $u\in A_x$. 
      
      By Claim \ref{claim:maximal surjective}, we know that $A_y\cap V_{y,x}^{-1}(u)$ is nonempty for all $y\in I$, $y\leq x$. 
      
      Define a family $B$ by 
      \[ B_y = \begin{cases}
         A_y \cap V_{y,x}^{-1}(u) & \mbox{ if } y\leq x,\\
         A_y & \mbox{ if } y>x.      
      \end{cases} \]
      Then $B$ lies in $\mathbf{S}$. Clearly $A\leq B$. By maximality of $A$, we have $A=B$; in particular $A_x=B_x$, hence $A_x=\{u\}$.
      Claim \ref{claim:maximal singleton} is proved.
   \end{proof}

   We can now finish the proof of Lemma \ref{lem:inverse limit}. By Claim \ref{claim:inverse limit upper bound}, we can apply Zorn's lemma to deduce that there exists $A\in \mathbf{S}$ which is maximal in $\mathbf{S}$. By Claim \ref{claim:maximal singleton}, we know that for each $x\in I$, there exists $u_x\in V_x$ such that $A_x = \{u_x\}$. Since $A$ satisfies the conditions in \eqref{eq:conditions on family}, it follows that $(u_x)$ satisfies the conditions in \eqref{eq:inverse limit}.
\end{proof}

\begin{deflemma}
   Let $P$ be a totally ordered set. Let $V$ be a pointwise finite dimensional persistence module indexed by $P$. Let $z\in P$. Then
   \begin{enumerate}
      \item $\left\{ \im(V_{x,z}) \mid x\in P,\ x\leq z \right\}$ is a flag in $V_z$, called the \emph{flag of images} of $V$ at $z$; 
      \item $\left\{ \ker(V_{z,x}) \mid x\in P,\ z\leq x \right\}$ is a flag in $V_z$, called the \emph{flag of kernels} of $V$ at $z$.
   \end{enumerate}
\end{deflemma}

\begin{proof}
(1) If $x\leq y\leq z$ in $P$, then $V_{x,z} = V_{y,z}\circ V_{x,y}$, hence $\im(V_{x,z}) \subseteq \im(V_{y,z})$. 

(2) If $z\leq x\leq y$ in $P$, then $V_{z,y} = V_{x,y}\circ V_{z,x}$, hence $\ker(V_{z,x}) \subseteq \ker(V_{z,y})$. 
\end{proof}

\begin{lemma} \label{lem:decomposition with maximal element}
   Let $P$ be a totally ordered set. Let $V$ be a pointwise finite dimensional persistence module indexed by $P$. Assume that $P$ has a maximal element $z$ and $V_z\neq 0$. Let $B$ be a basis of $V_z$ compatible with the flag of images of $V$ at $z$. Let $v\in B$. Then there exists an internal direct sum decomposition $V=U\oplus W$ such that:
   \begin{enumerate}
      \item $U$ is an interval module;
      \item $\{v\}$ is a basis of $U_z$;
      \item $B-\{v\}$ is a basis of $W_z$.
   \end{enumerate}
\end{lemma}

\begin{proof}
   Let 
   \[ I = \left\{ x \in P \mid V_{xz}^{-1}(v)\neq \emptyset \right\}. \] 
   Clearly $I$ is an interval. 
   
   By Lemma \ref{lem:inverse limit}, there exists $(u_x)\in \prod_{x\in I} V_x$ such that the conditions in \eqref{eq:inverse limit} are satisfied. Thus there exists a monomorphism $f: M(I) \to V$ such that for all $x\in I$,
   \[ f_x: k \to V_x, \qquad \lambda \mapsto \lambda u_x. \]
   Let $U=\im(f)$. Then $U$ is an interval module and $\{v\}$ is a basis of $U_z$.

   Let $D$ be the linear subspace of $V_z$ spanned by $B-\{v\}$. Let $W$ be the submodule of $V$ defined by
   \[ W_x = V_{xz}^{-1}(D) \quad \mbox{ for each } x\in P. \] 
   
   Then $B-\{v\}$ is a basis of $W_z$.

   It remains to see that $V_x = U_x \oplus W_x$ for each $x\in P$. There are two cases to consider:

   \textbf{Case 1}: $x\notin I$.

   Then we have $U_x=0$. Also,
   \begin{align*} 
      x\notin I \quad  &\Longrightarrow \quad v\notin \im(V_{x,z}) \\ 
      &\Longrightarrow \quad  B\cap \im(V_{x,z}) = \left(B-\{v\}\right) \cap \im(V_{x,z}).
   \end{align*} 
   Since $B$ is compatible with the flag of images of $V$ at $z$, it follows that 
   \[ \left(B-\{v\}\right)\cap \im(V_{x,z}) \] 
   is a basis of $\im(V_{x,z})$. Hence $\im(V_{x,z})\subseteq D$, which implies that $V_x=W_x$. Hence $V_x = U_x\oplus W_x$.

   \textbf{Case 2}: $x\in I$. 

   In this case, 
   \[ V_{x,z}(u_x) = v \quad \Longrightarrow \quad V_{x,z}(u_x) \notin D \quad \Longrightarrow \quad U_x \cap W_x = 0. \]
   But $V_{x,z}$ induces an injective map from $V_x/W_x$ to $V_z/D$, implying that
   \[ \dim(V_x/W_x) \leq 1. \] 
   Hence $V_x= U_x\oplus W_x$.  
\end{proof}

\begin{notation}
   For any poset $P$, we write $P^{\op}$ for the opposite poset of $P$.
\end{notation}

\begin{definition} \label{def:dual persistence module}
   Let $P$ be a poset. If $V$ is a persistence module indexed by $P$, then let $V^*$ be the persistence module indexed by $P^{\op}$ defined as follows:
   \begin{enumerate}
      \item $V^*_x = (V_x)^*$ for all $x\in P^{\op}$; 
      \item $V^*_{y,x} = {^t V_{x,y}}$ for all $x\geq y$ in $P^{\op}$.
   \end{enumerate}
\end{definition}

\begin{lemma} \label{lem:orthogonal module}
   Let $P$ be a poset. Let $V$ be a persistence module indexed by $P$. Let $\mathfrak{W}$ be a submodule of $V^*$. Then $V$ has a submodule $W$ defined by 
   \begin{equation} \label{eq:orthogonal module}
      W_x={^\perp\mathfrak{W}_x} \qquad \mbox{ for all }x\in P. 
   \end{equation}
   Moreover, we have $(V/W)^* \cong \mathfrak{W}$.
\end{lemma}

\begin{proof}
   It is immediate from Fact \ref{fact:transpose}(2) that \eqref{eq:orthogonal module} defines a submodule $W$ of $V$. Let $\pi : V\to V/W$ be the quotient map. By Fact \ref{fact:direct sum dual}, we know that for each $x\in P$, we have:
   \begin{enumerate}
      \item $\mathfrak{W}_x= W_x^\perp$,
      \item a natural bijective $k$-linear map
   \[ f_x: (V/W)_x^* \to W_x^\perp, \qquad \eta \mapsto \eta\circ \pi_x. \]
   \end{enumerate}
   We claim that this gives us an isomorphism $f: (V/W)^* \to \mathfrak{W}$.

   Let $x\geq y$ in $P^{\op}$. We need to check that 
   \[ f_x \circ (V/W)^*_{y,x} = V^*_{y,x} \circ f_y. \]
   For all $\theta\in (V/W)_y^*$, we have:
   \begin{align*}
            (f_x \circ (V/W)^*_{y,x})(\theta) 
            &= ( (V/W)^*_{y,x} (\theta)  ) \circ \pi_x \\ 
            &= \theta \circ (V/W)_{x,y} \circ \pi_x \\
            &= \theta \circ \pi_y \circ V_{x,y} \\ 
            &=  (f_y(\theta)) \circ V_{x,y} \\
            &= (V^*_{y,x} \circ f_y  )(\theta).
   \end{align*}
\end{proof}

\begin{lemma} \label{lem:dual basis is compatible}
   Let $P$ be a totally ordered set. Let $V$ be a pointwise finite dimensional persistence module indexed by $P$. Let $z\in P$ and $B$ a basis of $V_z$ compatible with the flag of kernels of $V$ at $z$. Let $B^*$ be the dual basis to $B$. Then $B^*$ is a basis of $V^*_z$ compatible with the flag of images of $V^*$ at $z$.
\end{lemma}

\begin{proof}
   Let $x\leq z$ in $P^{\op}$. We need to show that $B^*\cap \im(V^*_{x,z})$ is a basis of $\im(V^*_{x,z})$. By Fact \ref{fact:transpose}(1), we have 
   \[ \im(V^*_{x,z}) = \left(\ker(V_{z,x})\right)^\perp, \] 
   so we need to show that $B^*\cap \left(\ker(V_{z,x})\right)^\perp$ is a basis of $\left(\ker(V_{z,x})\right)^\perp$. This is obvious if $\ker(V_{z,x})=0$, so assume that $\ker(V_{z,x})\neq 0$. 
   
   We know that $B\cap \ker(V_{z,x})$ is a basis of $\ker(V_{z,x})$. Let $v_1, v_2, \ldots, v_n$ be the elements of $B$, and $v^*_1, v^*_2, \ldots, v^*_n$ the corresponding elements of $B^*$. Relabelling the elements of $B$ if necessary, we may assume that $v_1, \ldots, v_r$ (for some $r$) form a basis of $\ker(V_{z,x})$. Then $v^*_{r+1}, \ldots, v^*_n$ form a basis of $\left(\ker(V_{z,x})\right)^\perp$. This proves the desired result.
\end{proof}

\begin{lemma} \label{lem:decomposition with minimal element}
   Let $P$ be a totally ordered set. Let $V$ be a pointwise finite dimensional persistence module indexed by $P$. Assume that $P$ has a minimal element $z$ and $V_z\neq 0$. Let $B$ be a basis of $V_z$ compatible with the flag of kernels of $V$ at $z$. Let $v\in B$. Then there exists an internal direct sum decomposition $V=U\oplus W$ such that:
   \begin{enumerate}
      \item $U$ is an interval module;
      \item $\{v\}$ is a basis of $U_z$;
      \item $B-\{v\}$ is a basis of $W_z$.
   \end{enumerate}
\end{lemma}

\begin{proof}
   Let $B^*$ be the dual basis to $B$. Denote the elements of $B$ by $v_1, v_2, \ldots, v_n$ where $v_1=v$. Denote the corresponding elements of $B^*$ by $v^*_1, v^*_2, \ldots, v^*_n$.  
   
   By Lemma \ref{lem:dual basis is compatible}, we know that $B^*$ is compatible with the flag of images of $V^*$ at $z$. Thus by Lemma \ref{lem:decomposition with maximal element}, there exists an internal direct sum decomposition $V^* = \mathfrak{U} \oplus \mathfrak{W}$ such that:
   \begin{enumerate}[label=(\arabic*')]
      \item $\mathfrak{W}$ is an interval module;
      \item $\{v^*_1\}$ is a basis of $\mathfrak{W}_z$;
      \item $B^*-\{v^*_1\}$ is a basis of $\mathfrak{U}_z$.
   \end{enumerate} 

   By Lemma \ref{lem:orthogonal module}, there are submodules $U, W$ of $V$ such that 
   \[ U_x = {^\perp \mathfrak{U}_x}, \qquad W_x = {^\perp \mathfrak{W}_x} \qquad \mbox{ for all }x\in P; \]
   moreover, $(V/W)^* \cong \mathfrak{W}$. 
   
   By Fact \ref{fact:direct sum dual}(4), we have $V=U\oplus W$. Therefore 
   \[ U^* \cong (V/W)^* \cong \mathfrak{W}. \]
   It follows that $U\cong \mathfrak{W}^*$ and hence by (1') that $U$ is an interval module. 

   It is plain that:
   \begin{align*}
      (3') \quad &\Longrightarrow \quad \mbox{ $\{v_1\}$ is a basis of $U_z$; }\\
      (2') \quad &\Longrightarrow \quad \mbox{ $B-\{v_1\}$ is a basis of $W_z$. }
   \end{align*}
\end{proof}

\begin{lemma} \label{lem:gluing over totally ordered set}
   Let $P$ be a totally ordered set. Let $V$ be a nonzero pointwise finite dimensional persistence module indexed by $P$. Then there exists an internal direct sum decomposition $V=U\oplus W$ such that $U$ is an interval module.
\end{lemma}

\begin{proof}
   Choose and fix an element $z\in P$ such that $V_z\neq 0$. Define full subposets $P', P''$ of $P$ by
   \[ P'=\{x\in P \mid x\leq z\}, \qquad P''=\{x\in P \mid z\leq x\}. \]
   Then $P'$ is a totally ordered set with maximal element $z$, while $P''$ is a totally ordered set with minimal element $z$. Let $V', V''$ be the restrictions of $V$ to $P', P''$, respectively. 

   By Lemma \ref{lem:basis for two flags}, there exists a basis $B$ of $V_z$ which is compatible with both the flag of images and the flag of kernels of $V$ at $z$. 
   Choose such a basis $B$ and let $v\in B$.

   By Lemma \ref{lem:decomposition with maximal element}, there exists an internal direct sum decomposition $V'=U'\oplus W'$ such that $U'$ is an interval module, $\{v\}$ is a basis of $U'_z$, and $B-\{v\}$ is a basis of $W'_z$.

   By Lemma \ref{lem:decomposition with minimal element}, there exists an internal direct sum decomposition $V''=U''\oplus W''$ such that $U''$ is an interval module, $\{v\}$ is a basis of $U''_z$, and $B-\{v\}$ is a basis of $W''_z$.

   Observe that $U'_z=U''_z$ and $W'_z=W''_z$. This allows us to define submodules $U, W$ of $V$ by 
   \[
      U_x = \begin{cases}
         U'_x & \mbox{ if }x\leq z,\\
         U''_x & \mbox{ if }x>z;
      \end{cases}\qquad
      W_x = \begin{cases}
         W'_x & \mbox{ if }x\leq z,\\
         W''_x & \mbox{ if }x>z.
      \end{cases}
   \]
   Clearly $V=U\oplus W$. Moreover $U$ is an interval module.
\end{proof}

We can now deduce Theorem B.

\begin{theorem} \label{thm:totally ordered case}
   Let $P$ be a totally ordered set. Let $V$ be a pointwise finite dimensional persistence module indexed by $P$. Then $V$ has an interval decomposition.
\end{theorem}

\begin{proof}
   Immediate from Lemma \ref{lem:interval decomposition criteria} and Lemma \ref{lem:gluing over totally ordered set}.
\end{proof}

\section{Zigzag posets: generalities}  \label{sec: zigzag generalities}

\begin{notation} \label{nota:gamma}
   \begin{enumerate}
      \item We write $\mathbb{N}$ for the set $\{0, 1, 2, \ldots\}$.
      \item  We write $\Gamma$ for one of the following sets:
   \[ \{1, 2, \ldots, n\}, \qquad \mathbb{N}, \qquad \mathbb{Z} \]
   (where $n\geq 2$).
   \end{enumerate} 
\end{notation}

\begin{definition} \label{def:zigzag}
   Let $P$ be a poset. Let $\{z_i \mid i\in \Gamma\}$ be a sequence of elements of $P$ indexed by $\Gamma$ (see Notation \ref{nota:gamma}). 
   We call $P$ a \emph{zigzag poset} with extrema $\{z_i \mid i\in \Gamma\}$ if there exists a total order $\sqsubseteq$ on the set $P$ satisfying the following conditions:
   \begin{enumerate}
      \item for all $i,j\in \Gamma$:
      \[ i< j \quad \Longrightarrow \quad z_i\sqsubset z_j; \]
      \item for all $x,y\in P$:
      \begin{align*}
         x\leq y \quad \Longrightarrow \quad   
         & \mbox{ there exists $i\in \Gamma$ such that } \\
         & \begin{cases} 
               i+1\in \Gamma, \\
               z_i \sqsubseteq x\sqsubseteq z_{i+1}, \\ 
               z_i \sqsubseteq y\sqsubseteq z_{i+1};
            \end{cases} 
      \end{align*}
   \item either:
   \begin{align*}
      &\mbox{ for all $i\in \Gamma$ such that $i+1\in \Gamma$, }\\ 
      &\mbox{ for all $x,y\in P$ such that $z_i\sqsubseteq x \sqsubseteq y \sqsubseteq z_{i+1}$: }\\ 
      &\begin{cases} 
         i \mbox{ is odd } \quad \Longrightarrow \quad x\leq y, \\
         i \mbox{ is even } \quad \Longrightarrow \quad  y\leq x,
      \end{cases}  
   \end{align*} 
   or:
   \begin{align*}
      &\mbox{ for all $i\in \Gamma$ such that $i+1\in \Gamma$, }\\
      &\mbox{ for all $x,y\in P$ such that $z_i \sqsubseteq x \sqsubseteq y \sqsubseteq z_{i+1}$: }\\ 
      &\begin{cases}
         i \mbox{ is odd } \quad \Longrightarrow \quad y\leq x , \\
         i \mbox{ is even } \quad \Longrightarrow \quad x\leq y.
      \end{cases}
   \end{align*}   
   \end{enumerate}
   In this case, we call $\sqsubseteq$ the \emph{associated total order}.
\end{definition}

\begin{lemma} \label{lem:zigzag shape}
   Let $P$ be a zigzag poset with extrema $\{z_i \mid i\in \Gamma\}$. 
   Let $\sqsubseteq$ be the associated total order. Then we have the following. 
   \begin{enumerate}
      \item For all $x\in P$, there exists $i\in \Gamma$ such that $i+1\in \Gamma$ and $z_i\sqsubseteq x\sqsubseteq z_{i+1}$.
    
      \item For all $x,y\in P$ such that $x\leq y$, for all $i\in \Gamma$ such that  
      \[ x\sqsubseteq z_i \sqsubseteq y \quad \mbox{ or } \quad y\sqsubseteq z_i \sqsubseteq x, \] 
      we have $x=z_i$ or $y=z_i$.

      \item For all $x,y,z\in P$ such that $x\leq y\leq z$, we have 
      \[ x\sqsubseteq y\sqsubseteq z \quad \mbox{ or } \quad z\sqsubseteq y \sqsubseteq x. \]
   \end{enumerate}
\end{lemma}

\begin{proof}
   (1) Follows from condition (2) in Definition \ref{def:zigzag} when $x=y$. 

   (2) Suppose that $x\sqsubseteq z_i \sqsubseteq y$. 
   
   By condition (2) in Definition \ref{def:zigzag}, there exists $j\in \Gamma$ such that $j+1\in \Gamma$,
   \begin{gather*} 
      z_j \sqsubseteq x \sqsubseteq z_{j+1}, \\ 
      z_j \sqsubseteq y \sqsubseteq z_{j+1}. 
   \end{gather*}
   Hence we have  
   \[ z_j \sqsubseteq x \sqsubseteq z_i \sqsubseteq y \sqsubseteq z_{j+1}. \]
   By condition (1) in Definition \ref{def:zigzag}, we have $z_i = z_j$ or $z_i = z_{j+1}$. It follows that $z_i=x$ or $z_i=y$.

   Similarly if $y\sqsubseteq z_i \sqsubseteq x$.

   (3) By condition (2) in Definition \ref{def:zigzag}, there exists $i\in \Gamma$ such that $i+1\in \Gamma$,
   \begin{gather*}
            z_i \sqsubseteq x \sqsubseteq z_{i+1}, \\ 
            z_i \sqsubseteq z \sqsubseteq z_{i+1}. 
   \end{gather*}
   Then we also have $z_i \sqsubseteq y \sqsubseteq z_{i+1}$ because otherwise, by (2):
   \begin{align*}
      y \sqsubset z_i \quad &\Longrightarrow \quad x=z_i \mbox{ and } z=z_i\\
      &\Longrightarrow \quad y=z_i, \quad \mbox{ a contradiction; }
   \end{align*} 
   similarly,
   \begin{align*}
      z_{i+1} \sqsubset y \quad &\Longrightarrow \quad x=z_{i+1} \mbox{ and }z=z_{i+1} \\
      &\Longrightarrow \quad y=z_{i+1}, \quad \mbox{ a contradiction. }
   \end{align*}
   It is now easy to deduce the desired result from condition (3) in Definition \ref{def:zigzag}.
\end{proof}

\begin{lemma} \label{lem:interval in zigzag}
   Let $P$ be a zigzag poset with extrema $\{z_i \mid i\in \Gamma\}$. 
   Let $\sqsubseteq$ be the associated total order. Let $I$ be a nonempty subset of $P$. Then the following statements are equivalent:
   \begin{enumerate}
      \item $I$ is an interval in the poset $P$.
      \item For all $x,y,z\in P$ such that $x\sqsubseteq y\sqsubseteq z$, we have $y\in I$ whenever $x, z\in I$.
   \end{enumerate}
\end{lemma}

\begin{proof}
   (1)$\Longrightarrow$(2): Let $x,y,z\in P$ such that $x\sqsubseteq y\sqsubseteq z$ and $x,z\in I$. 
   
   Since $I$ is connected, there exists $y_0, y_1, \ldots, y_m\in I$ such that $y_0=x$, $y_m=z$, and 
   \[ y_{i-1} \leq y_i \enspace \mbox{ or } \enspace y_{i-1} \geq y_i \quad \mbox{ for all } i\in \{1, \ldots, m\}. \]
   Choose the smallest $i \in \{1, \ldots, m\}$ such that $y\sqsubseteq y_i$. Then we have $y_{i-1} \sqsubseteq y \sqsubseteq y_i$. 

   By condition (2) in Definition \ref{def:zigzag}, there exists $j\in \Gamma$ such that $j+1\in \Gamma$ and 
   \begin{gather*} 
      z_j \sqsubseteq y_{i-1} \sqsubseteq z_{j+1}, \\ 
      z_j \sqsubseteq y_i \sqsubseteq z_{j+1}.
   \end{gather*} 
   Hence we have $z_j \sqsubseteq y_{i-1} \sqsubseteq y \sqsubseteq y_i \sqsubseteq z_{j+1}$. 
   
   By condition (3) in Definition \ref{def:zigzag}, it follows that 
   \[ y_{i-1} \leq y \leq y_i \quad \mbox{ or } \quad y_i  \leq y \leq y_{i-1}. \]
   Since $I$ is convex, it follows that $y\in I$. 

   (2)$\Longrightarrow$(1): To prove that $I$ is convex, let $x, y, z\in P$ such that $x, z\in I$ and $x\leq y\leq z$. 

   By Lemma \ref{lem:zigzag shape}(3), we have 
   \[ x\sqsubseteq y \sqsubseteq z \quad \mbox{ or } \quad z \sqsubseteq y \sqsubseteq x. \]
   Since $x, z\in I$, we have $y\in I$. 

   To prove that $I$ is connected, let $x,z \in I$. 
   
   By Lemma \ref{lem:zigzag shape}(1), there exist $i, j\in \Gamma$ such that $i+1, j+1\in \Gamma$ and 
   \begin{gather*}
      z_i \sqsubseteq x \sqsubseteq z_{i+1},\\
      z_j \sqsubseteq z \sqsubseteq z_{j+1}.
   \end{gather*}
   If $i=j$, then by condition (3) in Definition \ref{def:zigzag}, we have $x\leq z$ or $x\geq z$ (in which case we are done). 
   
   Suppose that $i\neq j$. Without loss of generality we may assume $i<j$. By condition (1) in Definition \ref{def:zigzag}, we have 
   \[ x \sqsubseteq z_{i+1} \sqsubseteq z_{i+2} \sqsubseteq \cdots \sqsubseteq z_j \sqsubseteq z. \]
   It follows that $z_{i+1}, z_{i+2}, \ldots, z_j \in I$. Let
   \[ \begin{cases}  
      y_0=x, \\ 
      y_{j-i+1}=z, \\
      y_r = z_{i+r} \quad \mbox{ for all } r \in \{1, \ldots, j-i\}.
   \end{cases} \] 
   Then by condition (3) in Definition \ref{def:zigzag}, we have 
   \[ y_{r-1} \leq y_r \enspace \mbox{ or } \enspace y_{r-1} \geq y_r \quad \mbox{ for all } r\in \{1, \ldots, j-i+1\}. \] 
\end{proof}

\begin{lemma} \label{lem:extension by zero}
   Let $P$ be a zigzag poset with extrema $\{z_i \mid i\in \Gamma\}$. 
   Let $\sqsubseteq$ be the associated total order.
   Let $V$ be a persistence module indexed by $P$.
   \begin{enumerate}
      \item Let $i\in \Gamma$. 
      Let $P'$ be the full subposet of $P$ defined 
      \begin{align*}
         \mbox{ either by } \quad &P'=\{x\in P \mid x\sqsubseteq z_i\}, \\
         \mbox{ or by } \quad &P'=\{x\in P \mid z_i \sqsubseteq x\}. 
      \end{align*}
      Let $V'$ be the restriction of $V$ to $P'$. 
      Assume there exists an internal direct sum decomposition $V'=U'\oplus W'$ such that $U'$ is an interval module and $U'_{z_i}=0$. Then there exists an internal direct sum decomposition $V=U\oplus W$ such that $U$ is an interval module and $U'$ is the restriction of $U$ to $P'$. 

      \item Let $i,j\in \Gamma$ ($i<j$) and 
      let $P''$ be the full subposet of $P$ defined by 
      \[ P''=\{x\in P \mid z_i \sqsubseteq  x \sqsubseteq z_j\}. \]
      Let $V''$ be the restriction of $V$ to $P''$. 
      Assume there exists an internal direct sum decomposition $V''=U''\oplus W''$ such that $U''$ is an interval module and $U''_{z_i}=U''_{z_j}=0$. Then there exists an internal direct sum decomposition $V=U\oplus W$ such that $U$ is an interval module and $U''$ is the restriction of $U$ to $P''$. 
   \end{enumerate}
\end{lemma}

\begin{proof}
   (1) For all $x\in P$, let
   \[ U_x = \begin{cases}
      U'_x & \mbox{ if }x\in P',\\
      0 & \mbox{ if }x\notin P';
   \end{cases} \qquad  
   W_x = \begin{cases}
      W'_x & \mbox{ if }x\in P',\\
      V_x & \mbox{ if }x\notin P'.
   \end{cases} \]
   It is easy to see using Lemma \ref{lem:zigzag shape}(2) that this defines submodules $U, W$ of $V$. Clearly $V=U\oplus W$ and moreover $U$ is an interval module whose restriction to $P'$ is $U'$.

   (2) Similar to (1).
\end{proof}

\section{Zigzag posets: finite case} \label{sec: zigzag finite}

In this section, we prove Theorem C in the case when the zigzag poset has only finitely many extrema. Our proof uses ideas of Ringel in his proof of \cite{ri}*{Theorem 1.1}. 

\begin{theorem} \label{thm:finite zigzag}
      Let $P$ be a zigzag poset with extrema $\{ z_1, z_2, \ldots, z_n\}$, where $n\geq 2$. Let $V$ be a pointwise finite dimensional persistence module indexed by $P$. Then $V$ has an interval decomposition.
\end{theorem}

\begin{proof}
   We use induction on $n$. The base case $n=2$ is immediate from Theorem \ref{thm:totally ordered case}.
   
   Assume $n\geq 3$. 

   \begin{claim} \label{claim: finite case inductive step} 
      Let $V$ be a nonzero pointwise finite dimensional persistence module indexed by $P$. Then $V$ has an internal direct sum decomposition $V=U\oplus W$ such that $U$ is an interval module. 
   \end{claim}
   
   \begin{proof}[Proof of Claim \ref{claim: finite case inductive step}]
   Let $\sqsubseteq$ be the associated total order on $P$. 
   
   Define full subposets $P'$, $P''$, $Q$ of $P$ by 
   \begin{align*}
      P' &= \{ x\in P \mid z_2 \sqsubseteq x\}, \\
      P'' &= \{ x\in P \mid x \sqsubseteq z_{n-1} \},\\
      Q &= \{ x\in P \mid z_2\sqsubseteq x \sqsubseteq z_{n-1}\}.
   \end{align*}
   Let $V', V'', \overline{V}$ be the restrictions of $V$ to $P', P'', Q$, respectively. By induction hypothesis, there exist interval decompositions 
   \begin{equation} \label{eq:V' and V''}
   V' = \bigoplus_{U'\in \mathcal{A}} U', \qquad V'' = \bigoplus_{U''\in \mathcal{B}} U''. 
   \end{equation}

   There are 3 cases to consider. 

   \textbf{Case 1}: There exists $U'\in \mathcal{A}$ such that $U'_{z_2}=0$.

   In this case the claim follows from Lemma \ref{lem:extension by zero}(1).

   \textbf{Case 2}: There exists $U''\in \mathcal{B}$ such that $U''_{z_{n-1}}=0$. 

   In this case the claim follows again from Lemma \ref{lem:extension by zero}(1).

   \textbf{Case 3}: For all $U'\in \mathcal{A}$ and $U''\in \mathcal{B}$, we have $U'_{z_2}\neq 0$ and $U''_{z_{n-1}}\neq 0$. 

   For all $U\in \mathcal{A}\cup \mathcal{B}$, let $\overline{U}$ be the restriction of $U$ to $Q$. Then by \eqref{eq:V' and V''}, we have interval decompositions
   \[ \overline{V} =   \bigoplus_{U'\in \mathcal{A}} \overline{U'}, \qquad \overline{V} = \bigoplus_{U''\in \mathcal{B}} \overline{U''}.\]
   By Theorem \ref{thm:uniqueness}, there exists a bijection $\sigma: \mathcal{A} \to \mathcal{B}$ such that 
   \[ \sigma(U') = U'' \quad \Longrightarrow \quad \overline{U'} \cong \overline{U''}. \]

   Now let $U'\in \mathcal{A}$ and  $U'' =\sigma(U')$. Then 
   \[ \overline{U'}_{z_{n-1}} \cong \overline{U''}_{z_{n-1}} \neq 0. \]  
   Thus $\overline{U'}$ is an interval module indexed by $Q$ such that $\overline{U'}_{z_2}\neq 0$ and $\overline{U'}_{z_{n-1}}\neq 0$; by Lemma \ref{lem:interval in zigzag}, we therefore have $\overline{U'} \cong M(Q)$. Since $U'$ is an arbitrary member of $\mathcal{A}$, it follows that there exists an isomorphism
   \[ f: \overline{V} \xlongrightarrow{\sim} M(Q)^{\oplus d} \quad \mbox{ where } \quad d=|\mathcal{A}|. \]
   By Definition \ref{def:constant module}, for each $x\in Q$ we have $M(Q)_x=k$. Hence for each $x\in Q$,
   \[ f_x : V_x \xlongrightarrow{\sim} k^{\oplus d}. \]
   
   Define full subposets $R', R''$ of $P$ by 
   \[
      R' = \{ x\in P \mid x \sqsubseteq z_2\}, \qquad
      R'' = \{ x\in P \mid z_{n-1} \sqsubseteq x \}.
   \]
   Then $R', R''$ are totally ordered sets. Let $Z', Z''$ be the restriction of $V$ to $R', R''$, respectively. 

   Define a flag $\mathcal{F}$ in $V_{z_2}$ by 
   \[ \mathcal{F} = \begin{cases}
         \mbox{ the flag of images of $Z'$ at $z_2$ \enspace if \enspace $z_2$ is a maximal element in $R'$, }\\
         \mbox{ the flag of kernels of $Z'$ at $z_2$ \enspace if \enspace $z_2$ is a minimal element in $R'$. }\\
   \end{cases}\]
   Define a flag $\mathcal{G}$ in $V_{z_{n-1}}$ by 
   \[ \mathcal{G} = \begin{cases}
         \mbox{ the flag of images of $Z''$ at $z_{n-1}$ \enspace if \enspace $z_{n-1}$ is a maximal element in $R''$, }\\
         \mbox{ the flag of kernels of $Z''$ at $z_{n-1}$ \enspace if \enspace $z_{n-1}$ is a minimal element in $R''$. }\\
   \end{cases}\]
   Now define flags $f_{z_2}(\mathcal{F}), f_{z_{n-1}}(\mathcal{G})$ in $k^{\oplus d}$ by 
   \begin{align*}
            f_{z_2}(\mathcal{F}) &= \{ f_{z_2}(F) \mid F \in \mathcal{F} \},\\ 
            f_{z_{n-1}}(\mathcal{G}) &= \{ f_{z_{n-1}}(G) \mid G \in \mathcal{G} \}.
   \end{align*}

   By Lemma \ref{lem:basis for two flags}, there exists a basis $B$ of $k^{\oplus d}$ compatible with both  $f_{z_2}(\mathcal{F})$ and $f_{z_{n-1}}(\mathcal{G})$. 
   Choose such a basis $B$. 
   Then $f_{z_2}^{-1}(B)$ is a basis of $V_{z_2}$ compatible with $\mathcal{F}$, and $f_{z_{n-1}}^{-1}(B)$ is a basis of $V_{z_{n-1}}$ compatible with $\mathcal{G}$.

   Fix an element $v\in B$. 
   
   By Lemma \ref{lem:decomposition with maximal element}/\ref{lem:decomposition with minimal element}, there exists an internal direct sum decomposition 
   \[ Z'=K'\oplus W' \] 
   such that $K'$ is an interval module, $\{f_{z_2}^{-1}(v)\}$ is a basis of $K'_{z_2}$, and $f_{z_2}^{-1}(B-\{v\})$ is a basis of $W'_{z_2}$.

   By Lemma \ref{lem:decomposition with maximal element}/\ref{lem:decomposition with minimal element}, there exists an internal direct sum decomposition  
   \[ Z''=K''\oplus W'' \] 
   such that $K''$ is an interval module, $\{f_{z_{n-1}}^{-1}(v)\}$ is a basis of $K''_{z_{n-1}}$, and $f_{z_{n-1}}^{-1}(B-\{v\})$ is a basis of $W''_{z_{n-1}}$.

   Let $C, D$ be the linear subspaces of $k^{\oplus d}$ spanned by $\{v\}, B-\{v\}$, respectively. Define submodules $C(Q), D(Q)$ of $M(Q)^{\oplus d}$ by 
   \[ C(Q)_x = C, \qquad D(Q)_x = D \qquad \mbox{ for all }x\in Q. \]
   Then $C(Q)$ is an interval module and $M(Q)^{\oplus d} = C(Q) \oplus D(Q)$. 
   
   Let $\overline{C} = f^{-1}(C(Q))$ and $\overline{D} = f^{-1}(D(Q))$. Then $\overline{C}$ is an interval module and 
   \[ \overline{V} = \overline{C} \oplus \overline{D}. \]

   Observe that 
   \begin{align*}
      K'_{z_2} &= f_{z_2}^{-1}(C) = \overline{C}_{z_2}, & W'_{z_2} &= f_{z_2}^{-1}(D) = \overline{D}_{z_2}, \\
      K''_{z_{n-1}} &= f_{z_{n-1}}^{-1}(C) = \overline{C}_{z_{n-1}}, & W''_{z_{n-1}} &= f_{z_{n-1}}^{-1}(D) = \overline{D}_{z_{n-1}}.
   \end{align*}
   Thus there exist submodules $U, W$ of $V$ such that for all $x\in P$:
   \[ U_x = \begin{cases}
      K'_x & \mbox{ if } x\sqsubseteq z_2, \\
      \overline{C}_x & \mbox{ if } z_2 \sqsubset x \sqsubset z_{n-1}, \\
      K''_x & \mbox{ if } z_{n-1} \sqsubseteq x;
   \end{cases} \qquad
      W_x = \begin{cases}
      W'_x & \mbox{ if } x\sqsubseteq z_2, \\
      \overline{D}_x & \mbox{ if } z_2 \sqsubset x \sqsubset z_{n-1}, \\
      W''_x & \mbox{ if } z_{n-1} \sqsubseteq x.
      \end{cases} \]
   Clearly $V=U\oplus W$ and moreover $U$ is an interval module. 

   Claim \ref{claim: finite case inductive step} is proved.
   \end{proof}

   Claim \ref{claim: finite case inductive step} allows us to apply Lemma \ref{lem:interval decomposition criteria} to deduce that any pointwise finite dimensional module indexed by $P$ has an interval decomposition.
\end{proof}

\section{Monotonicity} \label{sec: monotonicity}

\begin{notation} \label{nota:zigzag truncate at n}
   Let $P$ be a zigzag poset with extrema $\{z_i \mid i\in \mathbb{N} \}$. 
   Let $\sqsubseteq$ be the associated total order. 
   Let $V$ be a persistence module indexed by $P$. 
   \begin{enumerate}
      \item    For each $n\in \mathbb{N}$, we write $P(n)$ for the full subposet of $P$ defined by 
      \[ P(n) = \{ x\in P \mid z_0 \sqsubseteq x \sqsubseteq z_n \}. \]
      \item  For each $n\in \mathbb{N}$, we write $V(n)$ for the restriction of $V$ to $P(n)$.
   \end{enumerate}
\end{notation}

\begin{definition} \label{def:monotonic}
   Let $P$ be a zigzag poset with extrema $\{z_i \mid i\in \mathbb{N} \}$. 
   Let $\sqsubseteq$ be the associated total order.
   Let $V$ be a pointwise finite dimensional persistence module indexed by $P$. 
   Let $m\in \mathbb{N}$. We say that $V$ is \emph{monotonic after $z_m$} if:
   \begin{quote}
      for every integer $n\geq m$, 

      for every interval decomposition $V(n) = \bigoplus_{U\in \mathcal{A}} U$,

      for every $U\in \mathcal{A}$ such that $U_{z_n}=0$,

      for every $x\in P$ such that $z_m \sqsubseteq x \sqsubseteq z_n$,
   \end{quote}
   we have $U_x = 0$.
\end{definition}

\begin{remark}
   Let $P$ be a zigzag poset with extrema $\{z_i \mid i\in \mathbb{N} \}$. 
   Let $V$ be a pointwise finite dimensional persistence module indexed by $P$. 
   Let $m, \ell\in \mathbb{N}$ such that $m<\ell$. If $V$ is monotonic after $z_m$, then $V$ is also monotonic after $z_\ell$.
\end{remark}

The proof of the following lemma is adapted from Botnan's proof of \cite{bot}*{Lemma 2.5}.

\begin{lemma} \label{lem:monotonic m exists}
   Let $P$ be a zigzag poset with extrema $\{z_i \mid i\in \mathbb{N} \}$. Let $V$ be a pointwise finite dimensional persistence module indexed by $P$. Assume that 
   \begin{quote}
     for every $n\in \mathbb{N}$, 
     
     for every interval decomposition $V(n) = \bigoplus_{U\in \mathcal{A}} U$,

      for every $U\in \mathcal{A}$,
   \end{quote}
  we have $U_{z_0} \neq 0$ or $U_{z_n} \neq 0$. Then there exists $m\in \mathbb{N}$ such that $V$ is monotonic after $z_m$.
\end{lemma}

\begin{proof}
   Let $\sqsubseteq$ be the associated total order.
   
   We know by Theorem \ref{thm:finite zigzag} that for any $n\in \mathbb{N}$, the persistence module $V(n)$ indexed by $P(n)$ has an interval decomposition. For any such interval decomposition, say $V(n)=\bigoplus_{U\in \mathcal{A}} U$, define
   \[ \mathcal{O}(\mathcal{A}) = \{U\in \mathcal{A} \mid U_{z_n}=0 \}. \]
   Observe that 
   \[ U \in \mathcal{O}(\mathcal{A}) \quad \Longrightarrow \quad U_{z_0} \neq 0; \]
   now, since $V_{z_0} = V(n)_{z_0} = \bigoplus_{U\in \mathcal{A}} U_{z_0}$, it follows that
   \[ |\mathcal{O}(\mathcal{A})| \leq \dim(V_{z_0}). \]
   We can therefore choose $m\in \mathbb{N}$ and an interval decomposition $V(m)=\bigoplus_{U\in \mathcal{B}} U$ such that $|\mathcal{O}(\mathcal{B})|$ is maximal. We shall show that $V$ is monotonic after $z_m$. 

   Let $n\geq m$ be any integer. Let 
   \[ V(n) = \bigoplus_{U\in \mathcal{A}} U \] 
   be any interval decomposition. 

   For each $U\in \mathcal{A}$, write $\overline{U}$ for the restriction of $U$ to $P(m)$. Let 
      \[ \overline{\mathcal{A}} = \{ \overline{U}  \mid U \in \mathcal{A}  \mbox{ and } \overline{U} \neq 0\}. \]
      Then $V(m) = \bigoplus_{T \in \mathcal{\overline{A}}} T$. 
      This is an interval decomposition of $V(m)$, hence by Theorem \ref{thm:uniqueness}, there exists a bijection $\sigma: \overline{\mathcal{A}} \to \mathcal{B}$ such that 
      \[ T \cong \sigma(T) \quad \mbox{ for all } \quad T\in \overline{\mathcal{A}}. \] 
      Thus 
      \begin{equation} \label{eq:OA OB}
         |\mathcal{O}(\overline{\mathcal{A}})| = |\mathcal{O}(\mathcal{B})|.
      \end{equation}

      Now let 
      \[ Q=\{x\in P \mid z_m \sqsubseteq x \sqsubseteq z_n\}. \]
      Observe that if $U\in \mathcal{A}$ and $\overline{U}\in \mathcal{O}(\overline{\mathcal{A}})$, then $\overline{U}\neq 0$ and $U_{z_m}=0$; since $U$ is an interval module, it follows from Lemma \ref{lem:interval in zigzag} that $U_x=0$ for all $x\in Q$. Therefore we have a bijection
      \[ \left\{ U \in \mathcal{O}(\mathcal{A}) \mid U_x = 0 \mbox{ for all }x\in Q \right\} \,\to\, \mathcal{O}(\overline{\mathcal{A}}), \qquad U \mapsto \overline{U}. \]
      Hence 
      \begin{align*}
         |\mathcal{O}(\overline{\mathcal{A}})| &= | \{ U \in \mathcal{O}(\mathcal{A}) \mid U_x = 0 \mbox{ for all }x\in Q \} | & \\
         &\leq |\mathcal{O}(\mathcal{A})| & \\
         &\leq |\mathcal{O}(\mathcal{B})| & \\
         &= |\mathcal{O}(\overline{\mathcal{A}})| & \mbox{ by \eqref{eq:OA OB}. }
      \end{align*}
      We deduce that $\{ U \in \mathcal{O}(\mathcal{A}) \mid U_x = 0 \mbox{ for all }x\in Q \} = \mathcal{O}(\mathcal{A})$. 
      In other words, if $U\in \mathcal{A}$ and $U_{z_n}=0$, then $U_x=0$ for all $x\in Q$. This proves the lemma.
\end{proof}

\section{Zigzag posets: infinite case} \label{sec: zigzag infinite}

\begin{lemma} \label{lem:infinite zigzag inductive step}
   Let $P$ be a zigzag poset with extrema $\{z_i \mid i\in \mathbb{N} \}$. Let $V$ be a pointwise finite dimensional persistence module indexed by $P$. Let $m\in \mathbb{N}$ and assume that $V$ is monotonic after $z_m$. Then 
  \begin{quote} 
      for every integer $n \geq m$, 

      for every internal direct sum decomposition $V(n)=Z \oplus W$ where $Z$ is an interval module and $Z_{z_n}\neq 0$,
  \end{quote}
  there exists an internal direct sum decomposition $V(n+1) = \accentset{\sim}{Z} \oplus \accentset{\sim}{W}$ such that: 
  \begin{enumerate}
      \item $\accentset{\sim}{Z}$ is an interval module and $\accentset{\sim}{{Z}}_{z_{n+1}} \neq 0$; 
      \item $Z, W$ are, respectively, the restrictions of $\accentset{\sim}{Z}, \accentset{\sim}{W}$ to $P(n)$.
  \end{enumerate}
\end{lemma}

\begin{proof}
   Let $\sqsubseteq$ be the associated total order.
   
   Let $n\geq m$ be any integer, let 
   \begin{equation} \label{eq:V equal U W}
      V(n)=Z\oplus W 
   \end{equation}
   be any internal direct sum decomposition where $Z$ is an interval module and $Z_{z_n}\neq 0$.
   
   By Theorem \ref{thm:finite zigzag}, there is an interval decomposition
   \[ V(n+1) = \bigoplus_{U \in \mathcal{A}} U. \]
   Let $P'$ be the full subposet of $P$ defined by 
   \[ P' = \{ x\in P \mid z_n \sqsubseteq x \sqsubseteq z_{n+1} \}. \]
   Let $V'$ be the restriction of $V(n+1)$ to $P'$. For each $U\in \mathcal{A}$, let $U'$ be the restriction of $U$ to $P'$. Let 
   \[ \mathcal{A}' = \{ U' \mid  U\in \mathcal{A} \mbox{ and } U'\neq 0 \}. \]
   Then we have an interval decomposition 
   \begin{equation} \label{eq: VV T}
      V' = \bigoplus_{T\in \mathcal{A}' }T. 
   \end{equation} 
   
   Since $V$ is monotonic after $z_m$, if $U\in \mathcal{A}$ and $U_{z_{n+1}}=0$, then $U'=0$. Therefore if $U\in \mathcal{A}$ and $U'\neq 0$, then $U_{z_{n+1}}\neq 0$. In other words,
   \begin{equation} \label{eq:T nonzero at right}
      T\in \mathcal{A}' \quad \Longrightarrow \quad T_{z_{n+1}}\neq 0.
   \end{equation} 
   We know that for each $T\in \mathcal{A}'$, there is an isomorphism $T \cong M(I^T)$ for some interval $I^T$ in $P'$. By \eqref{eq:T nonzero at right}, 
   \[  T\in \mathcal{A}' \quad \Longrightarrow \quad z_{n+1} \in I^T. \]

   Let
   \[ f: \bigoplus_{T\in \mathcal{A}'} T \xlongrightarrow{\sim} \bigoplus_{T\in \mathcal{A}'} M(I^T) \]   
   be an isomorphism such that $f(T) = M(I^T)$ for all $T\in \mathcal{A}'$. 
   Let 
   \[ \mathcal{A}'' = \{ T \in \mathcal{A}' \mid z_n \in I^T \}. \]
   Observe that 
   \begin{align*} 
      T\in \mathcal{A}'' \quad & \Longrightarrow \quad z_n \in I^T \; \mbox{ and } \; z_{n+1} \in I^T \\ 
      & \Longrightarrow \quad I^T = P'.
   \end{align*}
   On the other hand,   
   \begin{align*}
      T\in \mathcal{A}' - \mathcal{A}'' \quad & \Longrightarrow \quad z_n \notin I^T \\ 
      & \Longrightarrow \quad T_{z_n} = 0 \; \mbox{ and } \; M(I^T)_{z_n} = 0. 
   \end{align*} 
   Therefore 
   \[ f_{z_n} : \bigoplus_{T\in \mathcal{A}'} T_{z_n}  \xlongrightarrow{\sim} \bigoplus_{T\in \mathcal{A}''} M(P')_{z_n}. \]
   By \eqref{eq:V equal U W} and \eqref{eq: VV T}, we have 
   \[ \bigoplus_{T\in \mathcal{A}'} T_{z_n} = Z_{z_n} \oplus W_{z_n}, \] 
   hence
   \[ f_{z_n} : Z_{z_n} \oplus W_{z_n}  \xlongrightarrow{\sim} \bigoplus_{T\in \mathcal{A}''} M(P')_{z_n}. \]

   Let $C= f_{z_n}(Z_{z_n})$ and $D=f_{z_n}(W_{z_n})$. Then $\dim C= \dim Z_{z_n} = 1$, and
   \[ \bigoplus_{T\in \mathcal{A}''} M(P')_{z_n} = C \oplus D. \]  
   For every $x\in P'$, we have $M(P')_{z_n} = k = M(P')_x$. Thus for every $x\in P'$, we have 
   \[ \bigoplus_{T\in \mathcal{A}''} M(P')_x = C \oplus D. \] 
   Define submodules $C(P')$ and $D(P')$ of $\bigoplus_{T\in \mathcal{A}''} M(P')$ by 
   \[ C(P')_x = C, \qquad D(P')_x = D \qquad \mbox{ for all }x\in P'. \]
   Then $C(P')\cong M(P')$ and
   \[ \bigoplus_{T\in \mathcal{A}''} M(P') = C(P') \oplus D(P'). \] 

   Let $\accentset{\sim}{C} = f^{-1}(C(P'))$ and $\accentset{\sim}{D} = f^{-1}(D(P'))$. Then $\accentset{\sim}{C} \cong M(P')$ and  
   \[ \bigoplus_{T\in \mathcal{A}''} T = \accentset{\sim}{C} \oplus \accentset{\sim}{D}. \]
   Thus 
   \[  V' =  \accentset{\sim}{C} \oplus \accentset{\sim}{D} \oplus  \left( \bigoplus_{T\in \mathcal{A}'-\mathcal{A}''} T \right).  \]
   Since we have
   \begin{align*} 
      Z_{z_n} &= \accentset{\sim}{C}_{z_n}, \\
      W_{z_n} &= \accentset{\sim}{D}_{z_n}, \\ 
      T_{z_n} &= 0 \qquad \mbox{ for all } T\in \mathcal{A}'-\mathcal{A}'',
   \end{align*} 
   there exist submodules $\accentset{\sim}{Z}, \accentset{\sim}{W}$ of $V(n+1)$ such that for all $x\in P(n+1)$:
   \begin{align*}
      \accentset{\sim}{Z}_x &= \begin{cases}
         Z_x & \mbox{ if } x\sqsubseteq z_n, \\
         \accentset{\sim}{C}_x & \mbox{ if } z_n \sqsubset x \sqsubseteq z_{n+1};
      \end{cases} \\ 
      \accentset{\sim}{W}_x &= \begin{cases}
         W_x & \mbox{ if } x\sqsubseteq z_n, \\
         \accentset{\sim}{D}_x \oplus \left( \bigoplus_{T\in \mathcal{A}'-\mathcal{A}''} T_x \right) & \mbox{ if } z_n \sqsubset x \sqsubseteq z_{n+1}. 
      \end{cases}
   \end{align*}
   Clearly $V(n+1) = \accentset{\sim}{Z} \oplus \accentset{\sim}{W}$, and this is an internal direct sum decomposition with the desired properties.
\end{proof}

\begin{lemma} \label{lem:infinite zigzag monotonic case}
   Let $P$ be a zigzag poset with extrema $\{z_i \mid i\in \mathbb{N} \}$. Let $V$ be a pointwise finite dimensional persistence module indexed by $P$. Let $m\in \mathbb{N}$ and assume that $V$ is monotonic after $z_m$. Then for every internal direct sum decomposition 
   \[ V(m)=Z' \oplus W' \] 
   where $Z'$ is an interval module, there exists an internal direct sum decomposition 
   \[ V=Z\oplus W \]
   such that: 
  \begin{enumerate}
      \item $Z$ is an interval module;
      \item $Z', W'$ are, respectively, the restrictions of $Z, W$ to $P(m)$.
  \end{enumerate}
\end{lemma}

\begin{proof}
      If $Z'_{z_m}=0$, the result is immediate from Lemma \ref{lem:extension by zero}. 

      Suppose $Z'_{z_m}\neq 0$. Set $Z^{(m)}=Z'$ and $W^{(m)}=W'$. By Lemma \ref{lem:infinite zigzag inductive step}, we can choose inductively an internal direct sum decomposition
      \[ V(n) = Z^{(n)}\oplus W^{(n)} \quad \mbox{ for each integer } n>m \]
      such that: 
      \begin{enumerate}
         \item $Z^{(n)}$ is an interval module and $Z^{(n)}_{z_n} \neq 0$;
         \item $Z^{(n-1)}, W^{(n-1)}$ are, respectively, the restrictions of $Z^{(n)}, W^{(n)}$ to $P(n-1)$. 
      \end{enumerate}
      This allows us to define submodules $Z, W$ of $V$ by 
      \[ Z_x = Z^{(n)}_x, \qquad W_x = W^{(n)}_x \qquad \mbox{ if } x\in P(n). \]
      Clearly $V = Z\oplus W$, and this is an internal direct sum decomposition with the desired properties.
\end{proof}

\begin{lemma} \label{lem: infinite zigzag positive case}
   Let $P$ be a zigzag poset with extrema $\{z_i \mid i\in \mathbb{N} \}$. Let $V$ be a nonzero pointwise finite dimensional persistence module indexed by $P$. Then there exists an internal direct sum decomposition $V = Z\oplus W$ such that $Z$ is an interval module.
\end{lemma}

\begin{proof}
   We know by Theorem \ref{thm:finite zigzag} that for every $n\in\mathbb{N}$, the persistence module $V(n)$ indexed by $P(n)$ has an interval decomposition. There are two cases to consider: 

   \textbf{Case 1}: For some $n\in \mathbb{N}$, for some interval decomposition $V(n) = \bigoplus_{U\in \mathcal{A}} U$, for some $U\in \mathcal{A}$, we have $U_{z_n}=0$.

   In this case, the result is immediate from Lemma \ref{lem:extension by zero}.
   
   \textbf{Case 2}: For all $n\in \mathbb{N}$, for all interval decomposition $V(n) = \bigoplus_{U\in \mathcal{A}} U$, for all $U\in \mathcal{A}$, we have $U_{z_n}\neq 0$.

   In this case, by Lemma \ref{lem:monotonic m exists}, there exists $m\in \mathbb{N}$ such that $V$ is monotonic after $z_m$. We choose such an $m$ to be sufficiently large so that $V(m)$ is nonzero. 
   
   Let $V(m) = \bigoplus_{U\in \mathcal{A}} U$ be an interval decomposition. Let $Z'$ be any member of $\mathcal{A}$. Let $W'=\bigoplus_{U\in \mathcal{A}-\{Z'\}} U$. 
   Then $V(m)= Z' \oplus W'$. Moreover, $Z'$ is an interval module. The result now follows from Lemma \ref{lem:infinite zigzag monotonic case}.
\end{proof}

\begin{theorem} \label{thm: zigzag N case}
    Let $P$ be a zigzag poset with extrema $\{z_i \mid i\in \mathbb{N} \}$. Let $V$ be a pointwise finite dimensional persistence module indexed by $P$. Then $V$ has an interval decomposition.
\end{theorem}

\begin{proof}
   Immediate from Lemma \ref{lem:interval decomposition criteria} and Lemma \ref{lem: infinite zigzag positive case}.
\end{proof}

\begin{notation}
   Let $P$ be a zigzag poset with extrema $\{z_i \mid i\in \mathbb{Z} \}$. Let $\sqsubseteq$ be the associated total order. Let $n\in \mathbb{N}$. Let $V$ be a persistence module indexed by $P$.
   \begin{enumerate}
      \item We write $P^+, P^-, P(-n, n)$ for the full subposets of $P$ defined by 
      \begin{align*}
         P^+ &= \{ x\in P \mid z_0 \sqsubseteq x \}, \\
         P^- &= \{ x\in P \mid x \sqsubseteq z_0 \}, \\
         P(-n, n) &= \{ x\in P \mid z_{-n} \sqsubseteq x \sqsubseteq z_n \}.
      \end{align*} 

      \item We write $V^+, V^-, V(-n, n)$ for the restrictions of $V$ to $P^+, P^-, P(-n, n)$, respectively. 
   \end{enumerate}
\end{notation}

\begin{lemma} \label{lem:infinite zigzag two sided case}
   Let $P$ be a zigzag poset with extrema $\{z_i \mid i\in \mathbb{Z} \}$. Let $V$ be a nonzero pointwise finite dimensional persistence module indexed by $P$. Then there exists an internal direct sum decomposition $V = Z\oplus W$ such that $Z$ is an interval module.
\end{lemma}

\begin{proof}
   For each $n\in \mathbb{N}$, set $y_n = z_{-n}$. Observe that: 
   \begin{quote}
      $P^+$ is a zigzag poset with extrema $\{z_i \mid i\in \mathbb{N}\}$; 
      
      $P^-$ is a zigzag poset with extrema $\{y_i \mid i\in \mathbb{N}\}$. 
   \end{quote}   
   For each $n\in \mathbb{N}$, we write $V^+(n), V^-(n)$ for, respectively, the restrictions of $V^+, V^-$ to $P^+(n), P^-(n)$ (see Notation \ref{nota:zigzag truncate at n}).

   \textbf{Case 1}: For some $n\in \mathbb{N}$, for some interval decomposition $V^+(n) = \bigoplus_{U\in \mathcal{A}} U$, for some $U\in \mathcal{A}$, we have $U_{z_0}=U_{z_n}=0$.

   In this case, the result is immediate from Lemma \ref{lem:extension by zero}.

   \textbf{Case 2}: For some $n\in \mathbb{N}$, for some interval decomposition $V^-(n) = \bigoplus_{U\in \mathcal{A}} U$, for some $U\in \mathcal{A}$, we have $U_{y_0}=U_{y_n}=0$.

   In this case, the result is again immediate from Lemma \ref{lem:extension by zero}.
  
   \textbf{Case 3}: Neither Case 1 nor Case 2 occurs.

   In this case, by Lemma \ref{lem:monotonic m exists}, there exist $m, \ell \in \mathbb{N}$ such that: 
   \begin{quote}
      $V^+$ is monotonic after $z_m$;

      $V^-$ is monotonic after $y_\ell$.
   \end{quote}
   We choose an integer $n \geq \max\{ m, \ell \}$ so that $V^+(n)$ or $V^-(n)$ is nonzero.

   By Theorem \ref{thm:finite zigzag}, there is an interval decomposition $V(-n,n) = \bigoplus_{U\in \mathcal{A}} U$.
   Let $T$ be any member of $\mathcal{A}$. Let $S=\bigoplus_{U\in \mathcal{A}-\{T\}} U$. Then $T$ is an interval module, and $V(-n, n) = T\oplus S$. 

   Now let $Z', W'$ be, respectively, the restrictions of $T, S$ to $P^+(n)$. Then $Z'$ is zero or an interval module, and $V^+(n) = Z' \oplus W'$. Similarly let $Z'', W''$ be, respectively, the restrictions of $T, S$ to $P^-(n)$. Then $Z''$ is zero or an interval module, and $V^-(n) = Z'' \oplus W''$. 

   \begin{claim} \label{claim:left and right decomposition}
      \begin{enumerate}
         \item There exists an internal direct sum decomposition $V^+ = \accentset{\sim}{Z} \oplus \accentset{\sim}{W}$ such that:
            \begin{quote}
               $\accentset{\sim}{Z}=0$ if $Z'=0$;
               
               $\accentset{\sim}{Z}$ is an interval module if $Z'$ is an interval module; 

               $Z', W'$ are, respectively, the restrictions of $\accentset{\sim}{Z}, \accentset{\sim}{W}$ to $P^+(n)$.
            \end{quote}
         
         \item There exists an internal direct sum decomposition $V^- = \accentset{\approx}{Z} \oplus \accentset{\approx}{W}$ such that:
            \begin{quote}
               $\accentset{\approx}{Z}=0$ if $Z''=0$;
               
               $\accentset{\approx}{Z}$ is an interval module if $Z''$ is an interval module; 

               $Z'', W''$ are, respectively, the restrictions of $\accentset{\approx}{Z}, \accentset{\approx}{W}$ to $P^-(n)$.
            \end{quote}
      \end{enumerate}      
   \end{claim}
   
   \begin{proof}[Proof of Claim \ref{claim:left and right decomposition}]
      (1) If $Z'=0$, then let $\accentset{\sim}{Z}=0$ and $\accentset{\sim}{W}=V^+$.

      If $Z'\neq 0$, then $\accentset{\sim}{Z}$ and $\accentset{\sim}{W}$ with the desired properties exist by Lemma \ref{lem:infinite zigzag monotonic case}. 

      (2) Similar to (1).
   \end{proof}

   Let $\accentset{\sim}{Z}, \accentset{\sim}{W}, \accentset{\approx}{Z}, \accentset{\approx}{W}$ be chosen as in Claim \ref{claim:left and right decomposition}. We have: 
   \begin{equation} \label{eq:both sides join}
      \begin{split}
      \accentset{\sim}{Z}_{z_0} = Z'_{z_0} &= T_{z_0} = Z''_{z_0} = \accentset{\approx}{Z}_{z_0}, \\
      \accentset{\sim}{W}_{z_0} = W'_{z_0} &= S_{z_0} = W''_{z_0} = \accentset{\approx}{W}_{z_0}.
      \end{split}
   \end{equation} 
   Hence there exist submodules $Z, W$ of $V$ such that 
   \[ Z^+ = \accentset{\sim}{Z}, \qquad Z^- =\accentset{\approx}{Z}, \qquad W^+ = \accentset{\sim}{W}, \qquad W^-=\accentset{\approx}{W}. \]
   Clearly $V=Z\oplus W$. 

   It remains to see that $Z$ is an interval module. Since $T\neq 0$, we have $Z'\neq 0$ or $Z''\neq 0$, hence $\accentset{\sim}{Z}$ or $\accentset{\approx}{Z}$ is an interval module. 
   If $\accentset{\sim}{Z}=0$ or $\accentset{\approx}{Z}=0$, then $Z$ is obviously an interval module.
   If $\accentset{\sim}{Z}$ and $\accentset{\approx}{Z}$ are both interval modules, we need to see that $\accentset{\sim}{Z}_{z_0} \neq 0$ and $\accentset{\approx}{Z}_{z_0}\neq 0$. In this case, we have: 
   \begin{align*}
      & \mbox{ $\accentset{\sim}{Z}$ and $\accentset{\approx}{Z}$ are both interval modules } & \\
      & \Longrightarrow \quad Z'\neq 0 \enspace \mbox{ and } \enspace Z''\neq 0 & \\
      & \Longrightarrow \quad T_{z_0} \neq 0 & \mbox{ by Lemma \ref{lem:interval in zigzag} } \\
      & \Longrightarrow \quad \accentset{\sim}{Z}_{z_0} \neq 0 \enspace \mbox{ and } \enspace \accentset{\approx}{Z}_{z_0}\neq 0 & \mbox{ by \eqref{eq:both sides join}.}
   \end{align*} 
\end{proof}

\begin{theorem} \label{thm:zigzag Z case}
    Let $P$ be a zigzag poset with extrema $\{z_i \mid i\in \mathbb{Z} \}$. Let $V$ be a pointwise finite dimensional persistence module indexed by $P$. Then $V$ has an interval decomposition.
\end{theorem}

\begin{proof}
   Immediate from Lemma \ref{lem:interval decomposition criteria} and Lemma \ref{lem:infinite zigzag two sided case}.
\end{proof}

Theorem \ref{thm: zigzag N case} can be deduced from Theorem \ref{thm:zigzag Z case} but we give the proof of Theorem \ref{thm: zigzag N case} first since it is easier. 

Theorem \ref{thm: zigzag N case} and Theorem \ref{thm:zigzag Z case} give Theorem C for zigzag posets with infinitely many extrema.

\end{document}